\documentclass[12pt]{article}
 \usepackage{amsmath}
 \usepackage{blkarray}
  \usepackage{amsfonts}
\usepackage{tikz}
\usepackage[title]{appendix}
\usepackage{subfig}
\usepackage{array}
\usepackage{tabu}
\usepackage[
  separate-uncertainty = true,
  multi-part-units = repeat
]{siunitx}

\usetikzlibrary{arrows}
\usepackage{epsfig,}
\usepackage{epsfig}
\tikzset{
    vertex/.style = {
        circle,
        draw,
        outer sep = 3pt,
        inner sep = 3pt,
    },edge/.style = {->,> = latex'}
}

\usepackage{amssymb}
\usepackage{enumitem}
\usepackage{amsthm}
\usepackage{dsfont}
\def\diag{\mathop{\rm diag}}

\def\Diag{\mathop{\rm Diag}}

\def\rank{\mathop{\rm rank}}

\def\span{\mathop{\rm span}}

\newcommand{\rr}{\mathbb{R}}
\newcommand{\del}{\Delta(L^\dag)}
\newcommand{\1}{\mathbf{1}}
\newcommand{\g}{\mathcal{G}}

\def\c{\mathop{\rm col}}
\def\a{\alpha}
\def\b{\beta}
\def\col{{\col}}

\def\det{{\rm det}}
\def\min{{\rm min}}

\def\nulls {{\rm null}}

\def\S{{\widetilde{S}}}

\def\L{\widetilde{L}}
\def\W{\widetilde{W}}
\def\E{\mathcal{E}}

\def\G{\widetilde{G}}

\def\z{\mathbf{Z}}

\newtheorem{theorem}{Theorem}

\newtheorem{example}{Example}

\newtheorem{lemma}{Lemma}
\newtheorem{definition}{Definition}
\newtheorem{proposition}{Proposition}

\begin{document}
\begin{center}
\begin{large}
On resistance matrices of weighted balanced digraphs
\end{large}
\end{center}
\begin{center}
R. Balaji, R.B. Bapat and Shivani Goel \\
\today
\end{center}

\begin{abstract}
Let $G$ be a connected graph with $V(G)=\{1,\dotsc,n\}$.
Then the resistance distance between any two vertices $i$ and $j$ is given 
by $r_{ij}:=l_{ii}^{\dag} + l_{jj}^{\dag}-2 l_{ij}^{\dag}$,
where $l_{ij}^\dag$ is the $(i,j)^{\rm th}$ entry of the
Moore-Penrose inverse of the Laplacian matrix of $G$.
 For the resistance matrix 
$R:=[r_{ij}]$, there is an elegant formula to compute the inverse of $R$. This says that
\[R^{-1}=-\frac{1}{2}L + \frac{1}{\tau' R \tau} \tau \tau', \]
where \[\tau:=(\tau_1,\dotsc,\tau_n)'~~\mbox{and}~~
\tau_{i}:=2- \sum_{\{j \in V(G):(i,j) \in E(G)\}} r_{ij}~~~i=1,\dotsc,n. \]
A far reaching generalization of this result that gives an inverse formula for a generalized resistance matrix of a strongly connected and matrix weighted balanced directed graph is obtained in this paper. When the weights are scalars, it is shown that the generalized resistance
is a non-negative real number. We also obtain a perturbation result involving resistance matrices of connected graphs and Laplacians of digraphs.
\end{abstract}
\\
{\bf AMS CLASSIFICATION.} 05C50
\section{Introduction}
Let $G$ be a simple connected graph.
Suppose $x$ and $y$ are any two vertices of $G$. The length of the shortest path connecting
$x$ and $y$ in $G$
is the natural way to define the distance between $x$ and $y$.
This classical distance has certain limitations. For instance, consider two graphs $G_1$ and $G_2$ such that
\begin{enumerate}
\item[ \rm (i)] $V(G_1)=V(G_2)=\{1,\dotsc,n\}$.
\item[\rm (ii)] $i$ and $j$ are adjacent in both $G_1$ and $G_2$.
\item[\rm (iii)] There is only path between $i$ and $j$ in $G_1$ and there are multiple paths connecting $i$ and $j$
in $G_2$.
\end{enumerate}
Then the shortest distance between $i$ and $j$ in both $G_1$ and $G_2$ is {\it one}.
However, since there are multiple paths connecting $i$ and $j$ in $G_2$, the communication between $i$ and $j$ in
$G_2$ is better than in $G_1$. 
This significance is not reflected in the shortest distance.
Several applications require to overcome this limitation.
Instead of the classical distance, the so-called resistance distance
is used widely in many situations like in
electrical networks, chemistry and random walks: see 
for example \cite{bapat} and \cite{kr}.
If there are multiple paths between two vertices, then the resistance distance is less than
the shortest distance. The resistance matrix is now the matrix with $(i,j)^{\rm th}$ entry equal to the resistance distance between 
$i$ and $j$. Resistance matrices are non-singular and the inverse is given by an elegant formula that can be computed directly
from the graph. The main purpose of this paper is to deduce a formula for the inverse of a generalized resistance matrix of a simple digraph
with some special properties. 
This new formula generalizes the following known results. 
\subsection{Inverse of the resistance matrix of a connected graph} 
Let $G$ be a connected graph with $V(G)=\{1,\dotsc,n \}$. Let $\delta_i$ denote the degree of the vertex $i$ and $A$ be the
adjacency matrix of $G$. Then the Laplacian matrix of $G$ is
$L=\Diag(\delta_1,\dotsc,\delta_n)-A$.
Now the resistance between any two vertices $i$ and $j$ in $G$ is 
\begin{equation} \label{resistancedefn}
r_{ij}:=l_{ii}^{\dag} + l_{jj}^\dag -2l_{ij}^\dag,
\end{equation}
where $l_{ij}^\dag$ is the $(i,j)^{\rm th}$ entry of the Moore Penrose inverse of $L$.
Define $R:=[r_{ij}]$. Then the inverse of $R$ is given by
\begin{equation} \label{rinverseformula}
R^{-1}=-\frac{1}{2}L + \frac{1}{\tau' R \tau} \tau \tau', 
\end{equation}
where \[\tau:=(\tau_1,\dotsc,\tau_n)'~~\mbox{and}~~
\tau_{i}:=2- \sum_{\{j \in V(G):(i,j) \in E\}} r_{ij}~~~i=1,\dotsc,n. \]

The proof of (\ref{rinverseformula}) is given in Theorem 9.1.2 in \cite{bapat}.

\subsection{Inverse of the distance matrix of a tree}
Let $T$ be a tree with $V(T)=\{1,\dotsc,n\}$ and $r_{ij}$ (defined in (\ref{resistancedefn})) be
the resistance distance between any two vertices $i$ and $j$.
If $d_{ij}$ is the length of the shortest path connecting $i$ and $j$ in $T$, 
then by an induction argument, it can be shown that $d_{ij}=r_{ij}$. Define $D:=[d_{ij}]$.
Specializing formula (\ref{rinverseformula}) to $T$ gives
\begin{equation} \label{GrahamLovasz}
D^{-1}=-\frac{1}{2} L + \frac{(2-\delta_1,\dotsc,2-\delta_n)'(2-\delta_1,\dotsc,2-\delta_n)}{2 (n-1)}, 
\end{equation}
where $\delta_i$ is the degree of the vertex $i$ and $L$ is the Laplacian matrix of $T$.
This formula is obtained by Graham and  Lov\'{a}sz in \cite{Gr}.
\subsection{Inverse of the distance matrix of a weighted tree} \label{sec}
Formula $(\ref{GrahamLovasz})$ can be generalized to weighted trees. 
We first need to define the Laplacian matrix of a weighted tree.
Consider a tree $G=(V,\Omega)$ with
$V=\{1,\dotsc,n\}$. To an edge $(i,j) \in \Omega$, we assign a positive real number 
$w_{ij}$. Define

\begin{equation*}
    l_{ij} := \begin{cases}~~~~~~~~~ -\frac{1}{w_{ij}} & (i,j) \in \Omega \\ 
~~~~~~~~~~~~~~    0 & i \neq j~\mbox{and}~(i,j) \notin \Omega \\
   \sum \limits_{\{k:(i,k) \in \Omega\}} \frac{1}{w_{ik}} & i=j .
    \end{cases}
\end{equation*}

Then the Laplacian matrix of $G$ is $L:=[l_{ij}]$. 
The distance matrix of $G$ is the symmetric matrix $D$ with $(i,j)^{\rm th}$ entry equal to sum of all the weights that lie in the
path connecting $i$ and $j$.
In this case, by an induction argument, it can be shown that
$LDL+2L=0$ and from this identity it is easy to show that
$d_{ij}=l_{ii}^{\dag}+l_{jj}^{\dag}-2 l_{ij}^{\dag}$, where $l_{ij}^{\dag}$ is the $(i,j)^{\rm th}$ entry of the Moore Penrose inverse of $L$. 
Let $\delta_i$ be the degree of the vertex $i$.
In this setting, the following inverse formula is obtained in \cite{bkn}:
\begin{equation} \label{inverseweight}
D^{-1}=-\frac{1}{2}( L + \frac{\tau \tau'}{ \sum_{i,j} {w_{ij}}}),
\end{equation}
where $\tau$ is the vector $(2-\delta_1,\dotsc,2-\delta_n)'$.

\subsection{Inverse of the distance matrix of a tree with matrix weights} \label{secmw}
Formula $(\ref{inverseweight})$
can be generalized. Consider a tree on $n$ vertices with 
vertex set $V(T)=\{1,\dotsc,n\}$ and edge set $E(T)$. To an edge $(i,j)$ in $T$, assign a positive definite matrix $W_{ij}$ of some fixed order $s$.
Define
\begin{equation*}
    L_{ij} := \begin{cases} ~~~~~~~~~~~~-W_{ij}^{-1} & (i,j) \in E(T) \\ ~~~~~~~~~~~~~~~~~O_s & i \neq j~\mbox{and}~(i,j) \notin E(T) \\
    \sum \limits_{\{k: (i,k) \in E(T)\}} W_{ik}^{-1} & i=j.\\
    \end{cases}
\end{equation*}
(Here $O_s$ is the $s \times s$ matrix 
with all entries equal to zero.)
The Laplacian matrix $L$ of $T$ is then the $ns \times ns$ matrix with
$(i,j)^{\rm th}$ block equal to $L_{ij}$.
The distance between any two vertices $i$ and $j$ in $T$ is the sum of all positive definite matrices
that lie in the path connecting $i$ and $j$.
Let the $(i,j)^{\rm th}$ block of the Moore-Penrose inverse of $L$
be given by $M_{ij}$. Then, by induction it can be shown that
\[D_{ij}=M_{ii} + M_{jj} -2 M_{ij}. \]
The inverse of $D$ is 
\begin{equation} \label{matrixwinv}
D^{-1}=-\frac{1}{2}(L + FS^{-1} F'),
\end{equation}
where
\[S=\sum \limits_{i} \sum \limits_{j} W_{ij}~~\mbox{and}~~F=(2-\delta_1,\dotsc,2-\delta_n)' \otimes I_s. \]
Formula (\ref{matrixwinv}) is obtained in \cite{rbb}.

\subsection{Inverse of the resistance matrix of a directed graph}
Let $G=(V,\E)$ be a digraph with vertex set $V=\{1,\dotsc,n\}$. A directed edge from a vertex $i$ to a vertex $j$ in $G$ will be denoted by 
$(i,j)$. Recall that $G$ is said to be strongly connected if 
there is a directed path between any two vertices $i$ and $j$. A vertex $i$ is said to be balanced if the outdegree of $i$ and the indegree of $i$ are equal. If all the vertices are balanced, then we say that $G$ is balanced. We now assume that
$G$ is a strongly connected and balanced digraph.  
If $i$ and $j$ are any two vertices, define 
\begin{equation*}
    a_{ij} := \begin{cases} 1 & (i,j) \in \E \\ 0 & \text{otherwise.}
    \end{cases}
\end{equation*}
The Laplacian matrix of $G$ is the matrix $L=[l_{ij}]$ such that
\begin{equation*}
    l_{ij} := \begin{cases}~~~~~ -a_{ij} & i \neq j \\ \sum \limits_{\{k:k \neq i\}} a_{ik} & i=j. 
    \end{cases}
\end{equation*}
Now, let $L^{\dag}:=[l_{ij}^\dag]$ be the Moore-Penrose inverse of $L$.
Then the resistance between any two vertices $i$ and $j$ in $G$ is given by
\[r_{ij}:=l_{ii}^{\dag} + l_{jj}^{\dag} -2 l_{ij}^\dag. \]
For the resistance matrix $R=[r_{ij}]$ of $G$, we have the following inverse formula from
\cite{bbs}:
\begin{equation} \label{rinversedgraph}
R^{-1}=-\frac{1}{2} L + \frac{1}{\tau' R \tau} (\tau (\tau'+ \1' \diag(L^\dag)M)),
\end{equation}
where
\[M:=L-L',~~
\tau_i:=2-\sum_{\{j: (i,j) \in \E\}}r_{ji}, \]
and $\1$ is the vector of all ones in $\rr^n$.

We now ask if there is a formula that unifies
 (\ref{rinverseformula}), ($\ref{GrahamLovasz}$), ($\ref{inverseweight}$), ($\ref{matrixwinv}$) and ($\ref{rinversedgraph}$).

\subsection{Results obtained} \label{secr}
We consider a simple digraph $\g=(V,\E)$ with $V=\{1,\dotsc,n\}$.
An element in $\E$ will be denoted by $(i,j)$. Precisely, $(i,j) \in \E$ means that
there is a directed edge from a vertex $i$ to a vertex $j$ in $\g$.
All edges are assigned a positive definite matrix of some fixed order $s$.
These positive definite matrices will be called {\it weights}.
Let $W_{ij}$ be the weight of the edge $(i,j)$.
In this set up, we define the following.
\begin{enumerate}
\item[\rm (i)]  {\bf  Laplacian of $\g$:} If $i$ and $j$ are any two distinct vertices in $\g$,
define 
\begin{equation*}
    L_{ij} := \begin{cases} -W_{ij}^{-1} & (i,j) \in \E \\ ~~~~O_{s} & \text{otherwise}.
    \end{cases}
\end{equation*}
(Here, $O_s$ is the $s \times s$ matrix with all entries equal to zero.)
The Laplacian of $\g$ is then the $ns \times ns$ matrix 
\[L(\g):=
\left[
\begin{array}{rrrrr}
-\sum \limits_{\{j: j  \neq 1\}}  L_{1j} & L_{12} & \hdots & L_{1n} \\
L_{21} & -\sum \limits_{\{j: j \neq 2\}} L_{2j} & \hdots & L_{2n} \\
\vdots & \vdots  & \vdots & \vdots \\
L_{n1} & L_{n2} & \hdots & -\sum \limits_{\{j: j \neq n\}} L_{nj}
\end{array}
\right].
\]

We shall say that $L_{ij}$ is the $(i,j)^{\rm th}$ block of $L(\g)$. We note that since 
$\g$ is a digraph, $L(\g)$ is not symmetric in general.

\item[\rm (ii)] {\bf Resistance matrix of $\g$:} Let $i$ and $j$ be any two vertices in $\g$. Fix 
 $a,b>0$. Then, a { \it generalized  resistance distance} between $i$ and $j$ is the $s \times s$ matrix 
 \[R_{ij}=a^2K_{ii}+b^2 K_{jj} -2ab K_{ij}, \]
 where $K_{ij}$ is the $(i,j)^{\rm th}$ block of the Moore Penrose inverse of $L$.
The resistance matrix corresponding to $a$ and $b$ is then
 the $ns \times ns$
matrix with $(i,j)^{\rm th}$ block equal to $R_{ij}$.

 \item[\rm (iii)] {\bf Balanced vertices:} We say that a vertex $j \in V$ is balanced if
\[\sum_{\{i \in V:(i,j) \in \E\}} W_{ij}^{-1} = \sum_{\{i \in V:{(j,i) \in \E}\}} W_{ji}^{-1}. \]

\item[\rm (iv)] {\bf Balanced digraphs:} If every vertex of $\g$ is balanced, then we shall say that $\g$ is balanced.   
\end{enumerate}

We obtain the following result in this paper.
\begin{theorem} \label{T}
Let $\g=(V,\E)$ be a weighted, balanced and strongly connected
digraph, where $V=\{1,\dotsc,n\}$. Let $W_{ij}$ be the weight of
the edge $(i,j)$. 
Fix $a,b>0$.
If $R$ is the resistance matrix of $\g$,
$L$ is the Laplacian of $\g$, and $L^{\dag}$ is the Moore-Penrose inverse of $\g$,
then
\[ R^{-1} =- \frac{1}{2ab} L+  \tau (\tau' R \tau)^{-1} (\tau'-a^2U'\del L'+b^2 U' \del L),\]
where 
$\tau=(\tau_1,\dotsc,\tau_n)'$ is given by 
\[{\tau}_i := 2abI_s+L_{ii}R_{ii}-\sum_{\{j:(i,j) \in \E\}}W_{ij}^{-1}{R_{ji}},\]
$U = [\underbrace{I_s,\dotsc,I_s}_{n}]$ and $\del = \Diag(K_{11},\dotsc,K_{nn})$,  with $K_{ii}$ being the $i^{\rm th}$ diagonal block of $L^\dag$.
\end{theorem}

To illustrate, we give an example.
\begin{example}\rm
Consider the following graph $\g$ on four vertices.
\begin{figure}[tbhp]
\centering
\begin{tikzpicture}[shorten >=1pt, auto, node distance=3cm, ultra thick,
   node_style/.style={circle,draw=black,fill=white !20!,font=\sffamily\Large\bfseries},
   edge_style/.style={draw=black, ultra thick}]
\node[label=above:$1$,draw] (1) at  (0,0) {};
\node[label=above:$2$,draw] (2) at  (2,0) {};
\node[label=below:$3$,draw] (3) at  (2,-2) {};
\node[label=below:$4$,draw] (4) at  (0,-2) {};
\draw[edge]  (1) to (4);
\draw[edge]  (4) to (3);
\draw[edge]  (3) to (2);
\draw[edge]  (4) to (2);
\draw[edge]  (2) to (1); 
\end{tikzpicture}
\caption{Directed Graph $\g$.} \label{eg:dircycle}
\end{figure}
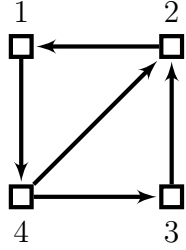 
Let
\[W_{14}= W_{21}=\frac{1}{5}\left[\begin{array}{rr}
3 & -4 \\
-4 & 7\end{array}\right],~~ W_{43}= W_{32}=\left[\begin{array}{rr}
1 & -1 \\
-1 & 2\end{array}\right],\]
and
\[W_{42}=\left[\begin{array}{rr}
2 & -3 \\
-3 & 5\end{array}\right].\]
The graph $\g$ is balanced with the above matrix weights. The Laplacian of $\g$ is
\[L = [L_{ij}]=\left[\begin{array}{rr|rr|rr|rr}
7 & 4 & 0 & 0 & 0 & 0 & -7 & -4 \\
4 & 3 & 0 & 0 & 0 & 0 & -4 & -3 \\
\hline
 -7 & -4 & 7 & 4 & 0 & 0 & 0 & 0 \\
-4 & -3 & 4 & 3 & 0 & 0 & 0 & 0 \\
\hline
 0 & 0 & -2 & -1 & 2 & 1 & 0 & 0 \\
0 & 0 & -1 & -1 & 1 & 1 & 0 & 0 \\
\hline
 0 & 0 & -5 & -3 & -2 & -1 & 7 & 4 \\
0 & 0 & -3 & -2 & -1 & -1 & 4 & 3\end{array}\right].\]
The Moore-Penrose inverse of $L$ is the matrix
\[L^{\dag} = [K_{ij}]=\frac{1}{80}\left[\begin{array}{rr|rr|rr|rr}
20 & -25 & -16 & 23 & -12 & 11 & 8 & -9 \\
-25 & 45 & 23 & -39 & 11 & -23 & -9 & 17 \\
\hline
8 & -9 & 20 & -25 & -24 & 27 & -4 & 7 \\
-9 & 17 & -25 & 45 & 27 & -51 & 7 & -11 \\
\hline
-12 & 11 & 0 & -5 & 36 & -33 & -24 & 27 \\
11 & -23 & -5 & 5 & -33 & 69 & 27 & -51 \\
\hline
-16 & 23 & -4 & 7 & 0 & -5 & 20 & -25 \\
23 & -39 & 7 & -11 & -5 & 5 & -25 & 45\end{array}\right].\]
Suppose $a=2$ and $b=3$. The resistance matrix of $\g$ is
\begin{equation*}
\begin{aligned}
    R  &= [4K_{ii}+9 K_{jj} -12 K_{ij}]\\
    &=\frac{1}{80}\left[\begin{array}{rr|rr|rr|rr}
20 & -25 & 452 & -601 & 548 & -529 & 164 & -217 \\
-25 & 45 & -601 & 1053 & -529 & 1077 & -217 & 381 \\
\hline
164 & -217 & 20 & -25 & 692 & -721 & 308 & -409 \\
-217 & 381 & -25 & 45 & -721 & 1413 & -409 & 717 \\
\hline
468 & -489 & 324 & -297 & 36 & -33 & 612 & -681 \\
-489 & 957 & -297 & 621 & -33 & 69 & -681 & 1293 \\
\hline
452 & -601 & 308 & -409 & 404 & -337 & 20 & -25 \\
-601 & 1053 & -409 & 717 & -337 & 741 & -25 & 45\end{array}\right].
\end{aligned}
\end{equation*}
Next, we compute $\tau=(\tau_1,\tau_2, \tau_3,\tau_4)'$. Recall that
\begin{equation*}\label{taudef}
{\tau}_i := 2abI_s+L_{ii}R_{ii}-\sum_{\{j:(i,j) \in E\}}W_{ij}^{-1}{R_{ji}}.
\end{equation*}
Thus,
\[\tau=\frac{1}{5}\left[\begin{array}{rr|rr|rr|rr}
15 & 0 & 15 & 0 & 21 & 2 & 9 & -2 \\
0 & 15 & 0 & 15 & 2 & 19 & -2 & 11\end{array}\right]',\]
\[\tau'R\tau =\frac{1}{5}\left[\begin{array}{rr}
2916 & -3159 \\
-3159 & 6075\end{array}\right],\]
and 
\[\tau'-4U'\del L'+9 U' \del L= \frac{1}{10}\left[\begin{array}{rr|rr|rr|rr}
30 & 0 & 3 & -9 & 57 & 9 & 30 & 0 \\
0 & 30 & -9 & 12 & 9 & 48 & 0 & 30\end{array}\right].\]
Hence, we have
\begin{equation*}
    \begin{aligned}
     &- \frac{1}{2ab} L+  \tau (\tau' R \tau)^{-1} (\tau'-4U'\del L'+9 U' \del L)\\
     \\
     &=\frac{1}{42444}\left[\begin{array}{rr|rr|rr|rr}
-23259 & -13368 & -84 & -138 & 3084 & 1698 & 26259 & 14928 \\
-13368 & -9891 & -138 & 54 & 1698 & 1386 & 14928 & 11331 \\
\hline
26259 & 14928 & -24843 & -14286 & 3084 & 1698 & 1500 & 780 \\
14928 & 11331 & -14286 & -10557 & 1698 & 1386 & 780 & 720 \\
\hline
2204 & 1188 & 6938 & 3351 & -2530 & -975 & 2204 & 1188 \\
1188 & 1016 & 3351 & 3587 & -975 & -1555 & 1188 & 1016 \\
\hline
796 & 372 & 17653 & 10521 & 8698 & 4371 & -23963 & -13776 \\
372 & 424 & 10521 & 7132 & 4371 & 4327 & -13776 & -10187
\end{array}\right]
    \end{aligned}
\end{equation*}
which is equal to $R^{-1}$.
\end{example}

\subsection{Other results}
We obtain the following two results after proving Theorem \ref{T}.
\begin{itemize}
\item By numerical computations, 
we observe that for any $a,b>0$, 
\[R_{ij}=a^{2} K_{ii} + b^{2} K_{jj}-2ab K_{ij} \]
is positive semidefinite. We do not know how to prove this result in general. However, when the weights in $\g$ of Theorem $1$ are
positive scalars, we show that $R_{ij}$ is always a non-negative real number.

\item Let $T$ be a tree on $n$ vertices. Suppose $D$ and $L$ are the distance and Laplacian matrices of $T$. Then, from $(\ref{GrahamLovasz})$ it can be deduced that
\[(D^{-1}-L)^{-1}=\frac{1}{3}D + \frac{1}{3} (\sum_{i,j} w_{ij} )\1 \1'.  \]
In particular, this equation says that every entry in 
$(D^{-1}-L)^{-1}$
is non-negative. Suppose $M$ is the Laplacian matrix of an arbitrary tree on $n$ vertices. 
It can be shown that $D^{-1}-M$ is non-singular.
We now say that
$(D^{-1}-M)^{-1}$ is a perturbation of the distance matrix $D$. In \cite{bkn}, it is shown that
all perturabations of $D$ are non-negative matrices. We now assume that $R$ is the resistance matrix of a
connected graph (defined in Section 1.1) on $n$ vertices. Now consider the Laplacian matrix $L$ of $\g$ in Theorem $\ref{T}$.  Suppose
all the weights in $\g$ are  positive scalars. 
It can be shown that $R^{-1}-L$ is always non-singular. We now say that $(R^{-1}-L)^{-1}$
is a perturabation of $R$.
By performing certain numerical experiments, we observed that similar to the result in \cite{bkn},  all perturbations of $R$
are non-negative matrices. Since $\g$ is a digraph, $L$ is not symmetric in general and hence all perturbations of $R$ are not symmetric matrices. 
Despite this difficulty, by using a argument different from \cite{bkn}, we show that all perturbations of $R$ are non-negative. This result is proved in the final part of this paper.
\end{itemize}

\section{Preliminaries}
We mention the notation and some basic results that will be used in the paper. 
\begin{enumerate}
\item[{\rm (i)}] We reserve $\g$ to denote a simple, strongly connected, weighted, and balanced digraph with vertex set $V=\{1,\dotsc,n\}$.
A directed edge from $i$ to $j$ in $\g$ will be denoted by $(i,j)$.
We use $\E$ to denote the edge set of $\g$.
 The weight of an edge $(i,j)$  will be denoted by 
$W_{ij}$. All weights will be symmetric positive definite matrices and have fixed order $s$.

\item[\rm (ii)] Let $B^{ns}$ be the set of all real $ns \times ns$ matrices.
A matrix $A$ in $B^{ns}$ will be denoted by $[A_{ij}]$, where $A_{ij}$ is an 
$s \times s$ matrix. We shall say that $A_{ij}$ is the $(i,j)^{\rm th}$ block of
$A$. There are $n$ blocks in $A$. The null-space of a matrix $A$ is denoted by $\nulls(A)$ and the
column space by $\c(A)$.

\item[\rm (iii)] The vector of all ones in $\rr^{n}$ will be denoted by $\1$. 
The matrix $\1' \otimes I_s$ will be denoted by $U'$, i.e.
\[U:=[I_s,\dotsc,I_s]', \]
where $I_s$ appears $n$ times. We use $J$ to denote the matrix in $B^{ns}$ with all blocks equal to
$I_s$. Note that $J=UU'$. 

\item[\rm (iv)] The Laplacian matrix of $\g$ will be denoted by $L$ and its Moore-Penrose by $L^{\dag}$. 
We note that $L$ and $L^{\dag}$ belong to $B^{ns}$. 
The $(i,j)^{\rm th}$ block of $L^\dag$ will be denoted by $K_{ij}$.
We use
$\Delta(L^\dag)$ to denote the block diagonal matrix
\[\mbox{Diag}(K_{11},\dotsc,K_{nn}).\] 
Let $\g'$ be the digraph 
such that $V(\g'):=\{1,\dotsc,n\}$ and 
\[E(\g'):=\{(j,i): (i,j) \in E(\g) \} \]
To an edge $(i,j)$ of $\g'$, we assign the weight $W_{ji}$.
Again $\g'$ will be strongly connected, and balanced.
The Laplacian of $\g'$ is clearly $L'$. 

\item[\rm (v)] 
Fix $a,b>0$.
The generalized resistance matrix of $\g$ corresponding to $a$ and $b$ will be denoted by $R_{a,b}$.
Thus, $R_{a,b}$ is an element in $B^{ns}$ with $(i,j)^{\rm th}$ block equal to
\[a^2 K_{ii} + b^{2} K_{jj} -2ab K_{ij}. \]

\item[\rm (vi)] 
Let $A$ be an $m \times m$ matrix. 
\begin{enumerate}
\item We say that $A$ is positive semidefinite if
$x'Ax \geq 0$ for all $x \in \rr^m$ and positive definite if $x'Ax >0$ for all $0 \neq x \in \rr^m$.
(To define positive semidefiniteness, we do not assume that $A$ is symmetric.)

\item Following \cite{Lewis}, we say that $A$ is almost positive definite if for each $x \in \rr^m$, either
$x'Ax>0$ or $Ax=0$. Suppose $A$ is almost positive definite. Then the Moore-Penrose inverse of $A$ is also
almost positive definite: see Corollary $2$ in \cite{Lewis}.
\end{enumerate}
\item[\rm (vii)] Let $B=[A_1, \dotsc,A_{n}]$, where each $A_{i}$ is an $s \times s$ matrix. 
As before, we say that $A_{j}$ is the $j^{\rm th}$ block of $B$.
Let $\Diag(B)$ be the 
$ns \times ns$ block matrix 
\[\Diag (A_{1},\dotsc,A_{n}).\]

\item[\rm (vii)] We use $[n]$ for $\{1,\dotsc,n\}$. The zero matrix of order $s \times s$ ($s \geq 2$) will be denoted by $O_s$. The identity matrix of order $k$ will be denoted by $I_k$.

\item[\rm (viii)]  Let $A=[a_{ij}]$ be an $m \times m$ matrix. We say that $A$ is row diagonally dominant if
\[|a_{ii}| \geq \sum_{\{j: j \neq i\}} |a_{ij}|.  \] We shall use the following well known result on diagonally dominant matrices.
\begin{theorem} \label{rdominant}
Let $A$ be row diagonally dominant. Suppose $A$ is non-singular. Let $B:=A^{-1}$ and $B=[b_{ij}]$.
Then, $$|b_{ii}| \geq |b_{ji}|~~\forall j. $$
\end{theorem}

\end{enumerate}

\section{Results}
To prove our main result, we need to show that the Laplacian of $\g$ has certain properties.
\subsection{Properties of the Laplacian}
From the digraph $\g=(V,\E)$, we define a simple undirected graph $\G$ as follows. 
\begin{definition} \rm \label{und}
Let $V(\G):=\{1,\dotsc,n\}$.
We say that any two vertices $i,j \in V(\G)$ are adjacent in $\G$
if and only if either $(i,j) \in \E$ or $(j,i) \in \E$.
\end{definition}

\begin{definition}  \rm
To an edge $(i,j)$ of $\G$, define
\begin{equation} \label{mmij}
    \W_{ij} := \begin{cases} (W_{ij}^{-1} + W_{ji}^{-1})^{-1} & (i,j) \in \E ~~\mbox{and}~~(j,i) \in \E\\ 
 ~~~~~~~~~~~~W_{ij} & (i,j) \in \E ~~\mbox{and}~~(j,i) \notin \E \\
 ~~~~~~~~~~~~ W_{ji}~~&(i,j) \notin \E~\mbox{and}~(j,i) \in \E. \\  
    \end{cases}
\end{equation} 
\end{definition}
We now have the weighted graph $\G$. Let $E$ be the set of all edges of $\G$.
\begin{definition} \rm \label{lap}
The Laplacian of $\G$ is the matrix $L(\G)=[M_{ij}] \in B^{ns}$, where
the $(i,j)^{\rm th}$ block is defined as follows:
 \begin{equation} \label{mij}
    M_{ij} := \begin{cases} ~~~~~~-\W_{ij}^{-1}  & (i,j) \in E \\ 
~~~~~~~~~~ O_s & i \neq j~~\mbox{and}~~ (i,j) \notin E\\
 \sum \limits_{\{k: k\neq i\} } \W_{ik}^{-1} & i=j. \\
 \end{cases}
\end{equation}
\end{definition}

We now have the following proposition. Recall that $L$ is the Laplacian of $\g$.
\begin{proposition} \label{ltilde}
 $L(\G)=L+L'$.
\end{proposition}
\begin{proof}
We shall write $$A=L+L'~~\mbox{and}~~ M=L(\G).$$
Let the $(i,j)^{\rm th}$ block of $A$ be $A_{ij}$ and $L$ be $L_{ij}$.
We need to show that \[A_{ij}=M_{ij} ~~\mbox{for all}~~ i,j \in V,\]
where $M_{ij}$ is defined in $(\ref{mij})$.
Note that $A_{ij}=L_{ij}+L_{ji}$. 
Partition the set of all edges $E$ of $\G$ as follows:
 \[S_1:=\{(i,j) \in E: (i,j) \in \E~~\mbox{and}~(j,i) \notin \E \}\]  
\[S_2:=\{(i,j) \in E: (i,j) \notin \E~\mbox{and}~(j,i) \in \E\} \]
\[S_3:=\{(i,j) \in E:(i,j) \in \E~\mbox{and}~(j,i) \in \E\}. \]
Fix $i$ and $j$ in $\{1,\dotsc,n\}$.
We consider the following cases.
 
{\bf Case 1:} Suppose $i$ and $j$ are not adjacent in $\G$. 
This means $(i,j) \notin \E$ and $(j,i) \notin \E$.
So, $L_{ij}=O_s$ and $L_{ji}=O_s$.
By $(\ref{mij})$, $M_{ij}=O_s$. 
Therefore, \[A_{ij}=L_{ij}+L_{ji}=O_s=M_{ij}.\]

  {\bf Case 2:} Suppose $i$ and $j$ are adjacent in $\G$.
We consider three possible sub-cases.

{\bf Case} (i): Let $(i,j) \in S_1$. 
Then, by $(\ref{mmij})$, $\W_{ij}=
W_{ij}$.
So, $M_{ij}=-W_{ij}^{-1}$.  
Because the weight of the edge $(i,j)$ in $\g$ is $W_{ij}$, $L_{ij}=-W_{ij}^{-1}$. 
Since $(j,i) \notin \E$, 
$L_{ji}=O_s$ and therefore,
\begin{equation} \label{1m}
 A_{ij}=L_{ij}+L_{ji}=-W_{ij}^{-1}=M_{ij}.
 \end{equation}
 
 {\bf Case} (ii): Let $(i,j) \in S_2$. 
Then, by (\ref{mmij}), $\W_{ij}=W_{ji}$.   
So, $M_{ij}=-W_{ji}^{-1}$.
As $(i,j) \notin \E$, $L_{ij}=O_s$. The weight of the edge $(j,i)$ in $\g$ is $W_{ji}^{-1}$. So, $L_{ji}=-W_{ji}^{-1}$; 
 and hence
\begin{equation} \label{2m}
 A_{ij}=L_{ij}+L_{ji}=-W_{ji}^{-1}=M_{ij}.
 \end{equation}
 
{\bf Case} 3: Let $(i,j) \in S_3$. 
Then, 
\[\W_{ij}=(W_{ij}^{-1}+W_{ji}^{-1})^{-1} .\]
So,
\[M_{ij}=-{\W_{ij}}^{-1}=-(W_{ij}^{-1}+W_{ji}^{-1}). \]
The weights of the edges $(i,j)$ and $(j,i)$ in $\g$ are respectively, $W_{ij}$ and $W_{ji}$. 
 So, \[L_{ij}=-W_{ij}^{-1} ~~\mbox{and}~~
 L_{ji}=-W_{ji}^{-1}.\]
 Thus,  
\begin{equation} \label{3m} 
 A_{ij}=L_{ij}+L_{ji}=
  -(W_{ij}^{-1} + W_{ji}^{-1})=M_{ij}.
  \end{equation}
Since  \[(L+L')U = MU = O_s ,\]
it follows that $A_{ii} = M_{ii}$ for each $i=1,\dotsc,n$. The proof is complete.
\end{proof}
We now deduce some properties of the Laplacian matrix $L$ of $\g$.
\begin{proposition}\label{rankL}
The following are true.
\begin{enumerate}
\item[\rm (i)] $L$ is positive semidefinite.
\item[\rm (ii)] $\nulls (L)=\nulls (L')=\c (J)$.
\item[\rm (iii)] $L^{\dag}$ is almost positive definite.
\item[\rm (iv)] $LL^\dag=L^\dag L=I_{ns}-\frac{J}{n}$.
\end{enumerate}
\end{proposition}
\begin{proof}
Consider the undirected graph $\G=(V,E)$ in Definition \ref{und}.
Put
$\L=L(\G)$ and
$S_{ij}=\W_{ij}^{-1}$, where $\W_{ij}$ are defined in $(\ref{mmij})$.
Corresponding to an edge $(p,q)$ in $\G$, we now define 
$\S(p,q) \in B^{ns}$ with $(i,j)^{\rm th}$ block given by
\begin{equation*}
    \S(p,q)_{ij} := \begin{cases} -S_{pq}&(i,j)=(p,q)~~\mbox{or}~~(i,j)=(q,p)  \\ 
~~ S_{pq}&i=j=p~~\mbox{or}~~i=j=q  \\
 ~~O_s &\mbox{else}.
 
    \end{cases}
\end{equation*}  
Now, 
\[\L=\sum_{(p,q) \in E} \S(p,q). \]
Let $x \in \rr^{ns}$.  Write
\[x=(x^1,\dotsc,x^j,\dotsc,x^n)', ~~~\mbox{where each}~~x^i \in \rr^s.\]
By an easy verification, we find that, if $(p,q) \in E$, then 
\[x'\S(p,q)x=\langle S_{pq}(x^p-x^q),x^p-x^q \rangle. \]
Thus,
\begin{equation} \label{Ltilde}
x'\L x=\sum_{(i,j) \in E} \langle S_{ij}(x^i-x^j),x^i-x^j \rangle. 
\end{equation}
Each $S_{ij}$ is a positive definite matrix. So, $x' \L x \geq 0$.
By Proposition \ref{ltilde}, $\L=L+L'$. Therefore, $x' \L x =2 x' L x$. So, $x'Lx \geq 0$.  
This proves (i).

Let $x \in \nulls(L)$.
As, $\L=L+L'$, we see that
$x'\L x=0$. 
By (\ref{Ltilde}), if $(i,j) \in E$, then $x^{i}=x^{j}$. 
Because $\g$ is strongly connected, 
$\G$ is connected. 
So,
$x^i=x^j$ for all $i,j$. Thus,
\[x \in \span\{(w,\dotsc,w)' \in \rr^{ns}: w \in \rr^s \}. \]
Since, $\c(J)=\span\{(w,\dotsc,w): w \in \rr^{s}\}$,
we see that $x=Jp$ for some $w \in \rr^{s}$. 
So, $\nulls(L)=\c(J)$.
Now, $\g'$ is a strongly connected, and balanced digraph. 
Because $L'$ is the Laplacian of $\g'$, we see that
$\nulls(L')=\c(J)$.
This proves (ii).

We now prove (iii).
Let $y \in \rr^{ns}$. Since $L$ is positive semidefinite, $y'Ly \geq 0$.
Suppose $y'Ly=0$. Then, $y' \L y=0$. 
By equation (\ref{Ltilde}), it follows that $y \in \c(J)$. Since $\c(J)=\nulls(L)$,
we have $Ly=0$. 
 Thus, either $y'Ly>0$ or $Ly=0$. So, $L$ is
almost positive definite. By item (vi) in Section $2$,
$L^\dag$ is almost positive definite as well. This proves (iii).

To prove (iv), we show that 
\[LL^\dag v=v~~\mbox{for all}~v \in \nulls(J). \]
Let $v \in \nulls(J)$.
Suppose $LL^\dag v=w$. 
Then, \[Jw=0
~~\mbox{and}~~L^\dag L L^\dag v= L^\dag w.\]
Since $L^\dag L L^\dag= L^\dag$, we get
 $L^\dag v=L^\dag w$
and hence $v-w \in \nulls(L^\dag)$. 
 As $\nulls(L^\dag)=\c(J)$, we get 
 $$v-w \in \c(J).$$ 
Since $JLL^\dag v=Jw$, and $JL=O_{ns}$, $Jw=0$. 
 As $v \in \nulls(J)$, $Jv=0$. So,
 $J(v-w)=0$ and hence \[v-w \in \nulls(J).\]
 We now have $v-w \in \nulls(J) \cap \c(J)$. So, $v=w$.
The proof is complete.
\end{proof}

\begin{proposition} \label{apd}
$\Delta(L^\dag)$ is a positive definite matrix.
\end{proposition}
\begin{proof}
We recall that $\Delta(L^\dag)=\Diag(K_{11},\dotsc,K_{nn})$.
Fix $i \in \{1,\dotsc,n\}$. We show that $K_{ii}$ is positive definite.
Let $y \in \rr^{s}$. 
Define $q:=(q^1,\dotsc,q^n)' \in \rr^{ns}$ by
\begin{equation*}
    q^j := \begin{cases}    
 y&j=i  \\
 0&\mbox{else}.
 \end{cases}
\end{equation*}  
In view of previous lemma, $L^\dag$ is almost positive definite. So, $q'L^{\dag}q >0$ or $L^\dag q=0$.
We note that $q'L^\dag q=y'K_{ii}y$. Hence, if $q' L^\dag q >0$, then $y'K_{ii}y >0$. 
Suppose $L^\dag q=0$. 
Since $\mbox{null}(L^\dag)=\mbox{null}(L')$ and $\mbox{null}(L')=\mbox{col}(J)$, 
$q=0$. This means 
$y$ is zero.
So, $K_{ii}$ is positive definite.  The proof is complete.
\end{proof}

\subsection{Inverse formula}
 Recall that the generalized resistance matrix of $\g$ corresponding to two positive real numbers $a$ and $b$ is 
 \[R_{a,b}:=[R_{ij}]=[a^2 L_{ii}^{\dag} + b^{2} L_{jj}^{\dag}-2ab L^{\dag}_{ij}].\]
Define
\begin{equation}\label{taudef}
{\tau}_i := 2abI_s+L_{ii}R_{ii}-\sum_{\{j:(i,j) \in \E\}}W_{ij}^{-1}{R_{ji}}~~\mbox{for all}~i=1,\dotsc,n.
\end{equation}
Now, let $\tau$ be the $ns \times s$ matrix $(\tau_1,\dotsc,\tau_n)'$.
The inverse formula will be proved by using the following lemma.
\begin{lemma}\label{3}
The following are true.
 \begin{enumerate}
\item[{\rm{(i)}}] $\tau = a^2 L\Delta(L^\dag)U + \frac{2ab}{n}U$.
\item[{\rm (ii)}] $\tau'+a^2 U' \Delta(L^{\dag})(L-L') = a^2 U' \Delta(L^{\dag})L+\frac{2ab}{n}U'$.
\item[{\rm (iii)}] $LR_{a,b}+2abI_{ns} = \tau U'$.
\item[{\rm (iv)}] $R_{a,b}L +2abI_{ns} = U\tau'-a^2UU'\Delta(L^{\dag}) L'+b^2UU' \Delta(L^{\dag})L$.
\item[{\rm (v)}] $U' \tau=2abI_s$.
\item[{\rm (vi)}] \(\tau'R_{a,b}\tau =2a^3b^3 \tilde{X}' L \tilde{X} + \frac{1}{\displaystyle n} 4a^2b^2(a^2+b^2) \sum_{i=1}^{n} K_{ii}\), where $\tilde{X}:=\Delta(L^{\dag})U.$
\item[{\rm (vii)}] $\tau'R_{a,b} \tau$ is a positive definite matrix.
\end{enumerate}
\end{lemma}

\begin{proof}
Fix $i \in [n]$. 
For simplicity, we shall use $R$ for $R_{a,b}$.  
Since $L=[L_{ij}]$ and $L^\dag=[K_{ij}]$,
the $(i,j)^{\rm th}$ block of $L L^\dag$ is the $s \times s$ matrix
\[L_{ii} K_{ii} - \sum_{\{j:(i,j) \in \E\}}W_{ij}^{-1}{K_{ji}} .\]
The $(i,j)^{\rm th}$ block of $I_{ns}-\frac{J}{n}$ is 
$(1-\frac{1}{n})I_{s}$.
Since $LL^{\dag}  = I_{ns}- \frac{1}{n} J$,
we see that
\[L_{ii} K_{ii} - \sum_{\{j:(i,j) \in \E\}}W_{ij}^{-1}{K_{ji}}  = (1-  \frac{1}{n})I_s.
\]
Rewriting the above equation, we have
\begin{equation} \label{two}
\sum_{\{j:(i,j) \in \E\}}W_{ij}^{-1}{K_{ji}} = L_{ii} K_{ii} -( 1 - \frac{1}{n})I_s.
\end{equation}
By definition,
\[\tau_i= 2abI_s+L_{ii}R_{ii}-\sum_{\{j:(i,j)\in \E\}}W_{ij}^{-1}{R_{ji}}.\]
Because 
\[R_{ji}=a^2K_{jj}+b^{2} K_{ii}-2ab K_{ji}~~\mbox{and}~~
R_{ii}=(a-b)^{2}K_{ii}, \]
we have 
\begin{equation} \label{tauii}
\tau_i=2abI_s+(a-b)^2L_{ii}K_{ii}- \sum_{\{j:(i,j) \in \E\}}W_{ij}^{-1}({a^2K_{jj} +b^2K_{ii}-2abK_{ji})}. 
\end{equation}
We recall that
\[L_{ii}=\sum_{\{j:(i,j) \in \E\}}W_{ij}^{-1}. \]
So,
\begin{equation} \label{cc}
\sum_{\{j:(i,j) \in \E\}} W_{ij}^{-1}K_{ii}=L_{ii}K_{ii}.
\end{equation}
Substituting $(\ref{cc})$  in $(\ref{tauii})$, 
\[\tau_i=2abI_s+(a-b)^2L_{ii}K_{ii}- b^2L_{ii}{K_{ii}} - a^2\sum_{\{j:(i,j) \in \E\}}W_{ij}^{-1}{K_{jj}}+2ab \sum_{\{j:(i,j) \in \E\}}W_{ij}^{-1}{K_{ji}} .\]
Using $(\ref{two})$ in the above equation, we get
\[\tau_i =2abI_s+(a-b)^2L_{ii}K_{ii}- b^2L_{ii}{K_{ii}} - a^2\sum_{\{j:(i,j) \in \E\}}W_{ij}^{-1}{K_{jj}}+ 2abL_{ii} K_{ii} -2ab( 1 - \frac{1}{n})I_s.\]
After simplification,

\begin{equation} \label{tau}
\begin{aligned} 
\tau_i&=
 a^2L_{ii}{K_{ii}} -a^2 \sum_{\{j:(i,j)\in \E\}}W_{ij}^{-1}{K_{jj}}+ \frac{2ab}{n}I_s .\\
 \end{aligned}
 \end{equation}
Let $A:= \Diag(L)-L$. Write $A=[A_{ij}]$. Then,
\begin{equation}\label{eqn8}
\begin{aligned}
\sum_{\{j:(i,j) \in \E\}} W_{ij}^{-1}K_{jj}&=\sum_{j=1}^n A_{ij} K_{jj} \\
&=(A\Delta(L^{\dag}) U)_i. \\
\end{aligned}
\end{equation}
We now compute $AU$. Because
\[A=
\left[
\begin{array}{cccccc}
O_s & -L_{12} & -L_{13} & \hdots & -L_{1n} \\
-L_{21} & O_s & -L_{23} & \hdots & -L_{2n} \\
\vdots & \vdots & \vdots & \vdots & \vdots \\
-L_{n1} & -L_{n2} & -L_{n3} & \hdots & O_{s}
\end{array}
\right],
\]
it follows that
\[(AU)_{i}=-\sum_{j \in [n] \smallsetminus \{i\}} L_{ij}. \]
Put 
\[P:=\Delta(L^\dag)U. \]
Then,
\[P_{i}=K_{ii}. \]
Thus,
\[(\Diag(AU)P)_{i}=-\sum_{j \in [n] \smallsetminus \{i\}} L_{ij}K_{ii}. \]
As
$\sum_{j=1}^{n} L_{ij}=O_s$, we get
\begin{equation} \label{au}
(\Diag(AU) P)_{i}=-\sum_{j \in [n] \smallsetminus \{i\}} L_{ij} K_{ii}=L_{ii}K_{ii}.
\end{equation}
By $(\ref{tau})$, $(\ref{eqn8})$ and $(\ref{au})$,
\begin{equation} \label{taui1}
\tau_{i}=a^{2}(\Diag(AU)P-AP)_{i} +\frac{2ab}{n} I_s.
\end{equation}
Put
\[\widetilde{A}:=\Diag(AU)-A. \]
In view of $(\ref{taui1})$,
\begin{equation} \label{taui12}
\tau_i=a^2 (\widetilde{A} P)_i +\frac{2ab}{n}I_s.
\end{equation}
But a direct verification tells that \[\widetilde{A}=L. \]
Therefore by $(\ref{taui12})$,
\[\tau=a^2 L \del U +\frac{2ab}{n} U. \]
This proves (i).

We now prove (ii).  Put $M:=L-L'$. Then by (i),
\begin{equation}\label{R}
\begin{aligned}
a^2 U' \Delta(L^{\dag}) M + \tau'&=a^2U' \Delta(L^{\dag}) L-a^2U' \Delta(L^{\dag})L'+a^2U' \Delta(L^{\dag}) L'+\frac{2ab}{n} U' \\
&=a^2U' \Delta(L^{\dag})L +\frac{2ab}{n} U'.
\end{aligned}
\end{equation}
The proof of (ii) is complete. 

We now prove (iii).
Since
\[R_{ij}=a^2 L^{\dag}_{ii} + b^{2} L^\dag_{jj}-2ab L^\dag_{ij}, \]
and $R=[R_{ij}]$, it is easy to see that
\[R=a^2\Delta(L^{\dag}) UU' + b^2 UU'\Delta (L^{\dag}) -2abL^{\dag}.\]
As $LL^{\dag}=I_{ns}-\frac{1}{n}UU'$ and $LU=O_s$, we get
\begin{equation}\label{eqn10}
\begin{aligned}
LR&=a^2 L \Delta(L^\dag)UU' -2ab LL^\dag \\ 
& =a^2L \Delta(L^\dag)UU' +\frac{2ab}{n}UU'-2ab I_{ns} \\
&=(a^2 L\Delta(L^\dag)U + \frac{2ab}{n})U'-2ab I_{ns}.
\end{aligned}
\end{equation}
By (i), 
\[\tau=a^2 L \Delta(L^\dag)U +\frac{2ab}{n} U. \]
Hence, \[LR= \tau U'-2ab I_{ns}.\] This completes the proof of (iii).

To prove (iv), first we observe that
\[RL=b^2 UU'\Delta(L^\dag)L-2abL^\dag L. \]
Since $L^{\dag}L=I_{ns}-\frac{1}{n}UU'$, we get
\begin{equation} \label{RL+2ab}
RL+2ab I_{ns}=b^2 UU' \Delta(L^\dag) L+\frac{2ab}{n} UU'.
\end{equation}
By (i),
\begin{equation}\label{Utau'}
\begin{aligned} 
U\tau' &=a^2 U U' \Delta(L^\dag) L' + \frac{2ab}{n} UU'. \\
\end{aligned}
\end{equation}
From $(\ref{RL+2ab})$ and $(\ref{Utau'})$,  
\[RL+2ab I_{ns}=U\tau' -a^2 UU' \Delta(L^\dag) L'+b^2 UU' \Delta(L^\dag) L.  \]
The proof of (iv) is complete.

By item (i),
\[U'\tau = a^2U'L\Delta(L^{\dag})U+\frac{2ab}{n}U'U .\] 
As $U'U=I_{ns}$ and $U'L=O_s$, it follows that
\[U' \tau =2ab I_{ns}.\]
This proves (v).

Put $M=L-L'$.
By (ii), 
\begin{equation} \label{tau'rtau}
\tau'R \tau=(a^2 U' \Delta(L^\dag)L + \frac{2ab}{n} U'-a^2 U' \Delta(L^\dag)M)R
\tau. 
\end{equation}
Because $M=L-L'$, 
\[a^2 U' \Delta(L^\dag)L + \frac{2ab}{n} U'-a^2 U' \Delta(L^\dag)M=a^2 U' \Delta(L^\dag) L'+\frac{2ab}{n}U'. \]
Substituting for $\tau$ from 
 (i) in  (\ref{tau'rtau}), we get
 \[\tau'R \tau=(a^2 U' \Delta(L^\dag) L'+\frac{2ab}{n}U')R (a^2 L \Delta(L^\dag)U + \frac{2ab}{n}U). \]
Therefore,
\begin{equation}\label{eqn16}
\begin{aligned}
\tau'R\tau 
&= a^4U'\Delta(L^{\dag})L'RL \Delta(L^{\dag})U + \frac{2a^3b}{n}U'\Delta({L^{\dag}})L'RU + \frac{2a^3b}{n}U'RL\Delta({L^{\dag}})U \\ &+ \frac{4a^2b^2}{n^2}U'RU. 
\end{aligned}
\end{equation}
As \[LU=L'U=O_s~~ \mbox{and}~~ R=a^2\del UU' + b^2 UU'\del -2abL^{\dag},\]
we have
\begin{equation*}\label{eqn17}
\begin{aligned}
U'\Delta(L^{\dag})L' R L \Delta(L^{\dag})U &=-2abU'\Delta(L^{\dag})L'L^{\dag}L\Delta(L^{\dag})U.
\end{aligned}
\end{equation*}
Since $L^{\dag}L=I_{ns}- \frac{1}{n}UU'$ and $L'U=O_s$, 
\begin{equation}\label{eqn18}
\begin{aligned}
U'\Delta(L^\dag)L'RL \Delta(L^{\dag})U &=  -2abU' \Delta(L^{\dag})L'\Big(I_{ns}-\frac{1}{n}UU'\Big)\Delta(L^{\dag})U \\ &= -2abU'\Delta(L^{\dag})L'\Delta(L^{\dag})U. 
\end{aligned}
\end{equation} 
Define \[\widetilde{X}:=\Delta(L^\dag)U. \]
Then,
\begin{equation} \label{xtilde}
\widetilde{X}'L \widetilde{X}=U' \Delta(L^\dag) L \Delta(L^\dag) U. 
\end{equation}
By (\ref{eqn18}),
\[U' \Delta(L^\dag) L' R L  \Delta(L^\dag) U=-2ab \widetilde{X}'L \widetilde{X}. \]
In view of (iv) and (i), we have
\[RL+2ab I_{ns}=U\tau' -a^2 UU' \Delta(L^\dag) L'+b^2 UU' \Delta(L^\dag) L;\]
\[U \tau'=a^2 UU' \Delta(L^\dag) L' + \frac{2ab}{n} UU'. \]
These two equations 
imply
\[RL=b^2UU' \Delta(L^{\dag})L+\frac{2ab}{n}UU'-2abI_{ns}. \]
Hence,
\begin{equation*}
\begin{aligned}
U'RL \del U &= U'\Big(b^2UU' \Delta(L^{\dag})L+\frac{2ab}{n}UU'-2abI_{ns}\Big)\Delta(L^{\dag})U \\ 
&= (b^2 n) U'\Delta(L^{\dag})L \Delta(L^{\dag})U. \\ 
\end{aligned}
\end{equation*}
By $(\ref{xtilde})$,
\begin{equation} \label{eqn19}
U' RL \del U=b^2 n \widetilde{X}' L \widetilde{X}. 
\end{equation}
We also note that
\begin{equation}\label{eqn20}
\begin{aligned}
U'\Delta(L^{\dag})L'RU &= U'\Delta(L^{\dag})L'(a^2 \Delta(L^{\dag})UU' + b^2UU' \Delta(L^{\dag}) -2abL^{\dag})U. \\
 &= a^2n U'\Delta(L^{\dag})L'\Delta(L^{\dag})U \\ 
 &= a^2n \widetilde{X}'L \widetilde{X},
\end{aligned}
\end{equation}
where the second equality follows from $L'U=L^\dag U=O_s$ and the last one from 
$(\ref{xtilde})$.
Since $$R=a^2 \Delta(L^\dag) UU'+ b^2 UU' \Delta(L^\dag)-2ab L^\dag,$$ we see that
\begin{equation}\label{eqn21}
U'RU  = n(a^2+b^2)\sum_{i=1}^{n}K_{ii}.
\end{equation}
Substituting $(\ref{eqn18})$, $(\ref{eqn19})$, $(\ref{eqn20})$ and $(\ref{eqn21})$ in $(\ref{eqn16})$, we get
\begin{equation*}
\begin{aligned}
\tau'R\tau =2a^3b^3 \widetilde{X}' L \widetilde{X} + \frac{4a^2b^2(a^2+b^2)}{n} \sum_{i=1}^{n}K_{ii}. 
\end{aligned}
\end{equation*}
The proof of (vi) is complete.

Since $L$ is positive semidefinite, $\widetilde{X}' L \widetilde{X} $ is positive semidefinite. By Proposition (\ref{apd}), each $K_{ii}$ is positive definite. 
So, $\tau' R \tau$ is positive definite.  
This proves (vii).
The proof is complete.
\end{proof}

We prove the inverse formula in Theorem \ref{T}.
\begin{theorem}\label{11}
\[ R_{a,b}^{-1} = - \frac{1}{2ab} L+  \tau (\tau' R \tau)^{-1} (\tau'-a^2U'\del L'+b^2 U' \del L).\]
\end{theorem}

\begin{proof}
Again, as in the proof of above lemma, we shall use $R_{a,b}$ for $R$.
By item (iii) of Lemma \ref{3},
\[LR + 2ab I_{ns}= \tau U'. \]
In view of item (v) of the previous Lemma, 
$U'\tau=2abI_s$. So,
\[LR \tau + 2ab\tau =\tau U' \tau=2ab \tau. \]
This implies \[LR \tau=O_s.\] 
We know that 
\[\nulls(L)=\span\{(p,\dotsc,p)':p \in \rr^{s}\}. \]
So, \[R \tau=UC, \]
where $C$ is a $s \times s$ matrix.
Since $\tau'R\tau$ is a positive definite matrix, $R\tau$ cannot be zero.
Hence, $C \neq O_s$.
As $\tau'U=2abI_s$, we get \[C = \frac{1}{2ab}\tau' R \tau.\] Therefore, 
\begin{equation} \label{trt}
R \tau= \frac{1}{2ab} U (\tau' R \tau).
\end{equation}
Since $L'U=O_s$, from item (iv) of Lemma $\ref{3}$, we deduce that
\begin{equation*}
(\tau' -a^2U'\del L'+b^2U' \del L) (RL+2abI_{ns})=
2ab(\tau' -a^2U'\del L'+b^2U'\del L).
\end{equation*}
Simplifying the above equation, we get
\begin{equation} \label{taue}
(\tau' -a^2U'\del L'+b^2U'\del L)RL=O_s.
\end{equation}
We now claim that 
\[(\tau' -a^2U'\del L'+b^2U'\del L)R \neq O_s.\]
If not, then 
\begin{equation} \label{tau'rtaus}
\tau' R \tau -a^2U' \del L'R\tau+b^2 U' \del LR \tau=O_s.
\end{equation}
By $(\ref{trt})$, 
\[R \tau =\frac{1}{2ab} U \tau'R \tau. \]
So, \[ L R \tau=O_s~~\mbox{and}~~L' R \tau =O_s.\] 
Hence (\ref{tau'rtaus}) leads to $\tau'R \tau=O_s$. 
This contradicts that $\tau'R \tau$ is positive definite.
Hence, \[(\tau' -a^2U'\del L'+b^2 U' \del L) R \neq O_s.\]
Since nullity of $L'$ is $s$ and $L'U =O_s$, by (\ref{taue}), there exists an $s \times s$ matrix $\widetilde{C}$
such that
\[(\tau' -a^2U' \del L'+b^2 U' \del (L^{\dag})L)R=\tilde{C}U'.\]
We know that $U' \tau=2abI_s$. So, from the previous equation,
 \[\tilde{C}= \frac{1}{2ab}\tau' R \tau .\]  Thus,
\begin{equation} \label{q}
(\tau' -a^2U'\del L'+b^2 U' \del L)R = \frac{\tau' R \tau}{2ab}U'.
\end{equation}
We now have 
\begin{equation*}
\begin{aligned}
Q:=&\Big(- \frac{1}{2ab} L+  \tau (\tau' R \tau)^{-1} (\tau'-a^2U' \del L'+b^2U'\del L)\Big)R \\&= -\frac{1}{2ab} LR+  \tau(\tau' R \tau)^{-1} ({\tau}'-a^2 U' \del L'+b^2 U' \del L)R.  \\
\end{aligned}
\end{equation*}
By (\ref{q}), we have 
\begin{equation} \label{Q}
Q=-\frac{1}{2ab} (LR - \tau U'). 
\end{equation}
Item (iii) of Lemma \ref{3} says that
\[L R + 2ab I_{ns}=\tau U'. \]
Substituting back in (\ref{Q}), we get  $Q=I_{ns}$.
The proof is complete.
\end{proof}

\subsection{Special cases}
\begin{enumerate}
\item[\rm (i)] Suppose all the weights in $\g$ are equal to $1$. 
Choose $a=b=1$. 
We shall denote $R_{ij}$ by $r_{ij}$ and define $R:=[r_{ij}]$.
Now by $(\ref{taudef})$,
\[\tau_{i}=2-\sum_{\{j: (i,j) \in \E\}} r_{ji}. \]
We note that $r_{ii}=0$ and
$U=\1$. Hence, by our formula in Theorem $\ref{T}$,
\begin{equation*}
\begin{aligned}
R^{-1} &= - \frac{1}{2} L+  \tau (\tau' R \tau)^{-1} (\tau'-\1'\del L'+ \1' \del L) \\
&= -\frac{1}{2} L+ \frac{\tau}{\tau'R \tau}(\tau'-\1' \del (L-L')).  
\end{aligned}
\end{equation*}
Thus we get $(\ref{rinversedgraph})$.

\item[\rm (ii)] Suppose $T$ is a tree with $V(T)=\{1,\dotsc,n\}$.
To denote an edge in $T$, we shall use the notation $ij$. 
Let the weight of an edge $ij$ be
$W_{ij}$. Assume that all weights are positive definite matrices of order $s$.
Now, define a directed graph $\widetilde{T}$ as follows. Let
$V(\widetilde{T})=\{1,\dotsc,n\}$. 
We use the notation $(i,j)$ to denoted a directed edge from $i$ to $j$.
Now, we define $E(\widetilde{T}):=\{(i,j),(j,i): ij \in E(T) \}$.
Now we assign the weight $W_{pq}$ to an edge $(p,q)$
in $\widetilde{T}$. It is clear that $\widetilde{T}$
is strongly connected, weighted and balanced digraph. 
Now define the Laplacian matrix of $T$, say, $L(T)$ 
as given in \ref{secmw} and the 
Laplacian of $\widetilde{T}$, say, $L(\widetilde{T})$ as given in 
item (i) of \ref{secr}. We note that $L(T)=L(\widetilde{T})$.
Fix $a=b=1$. Let
$R_{ij}$ be the resistance between $i$ and $j$.
Then, 
\[R_{ij}=M_{ii}+M_{jj}-2M_{ij}, \]
where $M_{ij}$ is the $(i,j)^{\rm th}$ block
of the Moore Penrose inverse of $L(\widetilde{T})$.
If $D_{ij}$ is the shortest distance between $i$ and $j$ 
in $T$, then by the argument mentioned in \ref{secmw}, 
$D_{ij}=R_{ij}$. Define $D:=[D_{ij}]$. 
Because $D_{ii}=O_s$, by Theorem \ref{T}, we have 
\begin{equation*}
\begin{aligned}
{\tau}_i : & = 2 I_s - \sum_{\{j:(i,j) \in E\}}W_{ij}^{-1}{R_{ji}} \\
&=2 I_s - \sum_{\{j:(i,j) \in E\}}W_{ij}^{-1}{W_{ij}} \\
&=(2-\delta_i)I_{s},
\end{aligned}
\end{equation*}
where $\delta_i$ is the out-degree of the vertex $i$.
We now define
\[\tau:=(2-\delta_1,\dotsc,2-\delta_n)' \otimes I_s. \]
By an induction argument it follows that if $S$ is the sum of all the weights in $T$, then
\[D \tau=[S,\dotsc,S]'. \]
(See Lemma $1$ in \cite{bs}). Because $T$ is a tree, $\sum_{i=1}^{n} \delta_i=2(n-1)$.
Thus, $\sum_{i=1}^{n} \tau_i=2 I_s$ and hence 
$\tau'D \tau=2S$. 
By our formula in Theorem \ref{T}, 
 \begin{equation}
\begin{aligned} 
 D^{-1}&=-\frac{1}{2} L(T) + \tau (\tau' D \tau)^{-1}\tau' \\
 &=-\frac{1}{2} L(T)
 + \frac{1}{2} \tau S \tau'.
  \end{aligned}
 \end{equation}
 This is formula  (\ref{matrixwinv}).
 
 In a similar manner, we get $(\ref{rinverseformula})$, $(\ref{GrahamLovasz})$, and $(\ref{inverseweight})$. 
\end{enumerate}

\subsection{Non-negativity of the resistance}
From numerical computations, we observe that $R_{ij}=a^2 l_{ii}^{\dag} + b^{2} l_{jj}^{\dag} -2ab l_{ij}^{\dag}$ is always 
positive semidefinite. But at this stage, we do not know how to prove this. However, when all the weights in $\g$ are positive scalars,
we now show that the resistance is always non-negative.
We need the following lemma.

\begin{lemma} \label{lapprop}
Let $A=[a_{ij}]$ be an $n \times n$ matrix with the following properties.
\begin{enumerate}
\item[\rm (i)] All the off-diagonal entries are non-positive.
\item[\rm (ii)] $A\1=A'\1=0$.
\item[\rm (iii)] $\rank(A+A')=n-1$.
\item[\rm (iv)] $A$ is positive semidefinite.
\end{enumerate}
Let $A^\dag:=[p_{ij}]$ be the Moore-Penrose inverse of $A$. Then,
\[p_{ii} \geq p_{ij}~~\mbox{and}~~p_{ii} \geq p_{ji}~~\forall j.  \]
\end{lemma}
\begin{proof}
By a permutation similarity argument, without loss of generality, we may assume that $i=1$ and $j=n$. 
We now show that
$p_{11} \geq p_{1n}$ and $p_{11} \geq p_{n1}$.
By symmetry of our assumptions, it suffices to show that
$p_{11} \geq p_{1n}$.

Let $\1_{n-1}$ be the vector of all ones in $\rr^{n-1}$.
By (ii) we can partition $A$ as follows:
\[A=
\left[ 
\begin{array}{ccccc}
B & -B \1_{n-1} \\
-\1_{n-1}'B & \1'_{n-1} B \1_{n-1} \\
\end{array}
\right].
\]
All the row sums of $A+A'$
are equal to zero. So, all the cofactors of $A+A'$
are equal. As $\rank(A+A')=n-1$, we now deduce that the common cofactor of $A+A'$
is non-zero. In particular, $\det(B+B') \neq 0$. Since $A$ is positive semidefinite, 
$B+B'$ is positive semidefinite. Because $B+B'$ is non-singular, $B+B'$ is positive definite. 
So, $B$ is positive definite. All the off-diagonal entries of $B$ are non-positive. By a well-known theorem
on $\z$-matrices, $B$ is non-singular and all entries of $B^{-1}$ are non-negative.
By a direct verification, 
\begin{equation} \label{pdag}
A^\dag= \left[
\begin{array}{cccc}
B^{-1} - \frac{1}{n} \1_{n-1} \1_{n-1}' B^{-1} - \frac{1}{n} B^{-1} \1_{n-1} \1_{n-1}' & -\frac{1}{n} B^{-1} \1_{n-1} \\
-\frac{1}{n}\1_{n-1}'B^{-1} & 0
\end{array}
\right]
+ \frac{\1_{n-1}'B^{-1} \1_{n-1}}{n^2}(\1 \1') .
\end{equation}
Put \[C=[c_{ij}]:=B^{-1} ~~\mbox{and}~~ \delta:=\frac{1}{n^2} \1_{n-1}'B^{-1}\1_{n-1}.\]
Then, $c_{ij} \geq 0 ~~\forall i,j$ and
\begin{equation*}
\begin{gathered}
p_{11}=c_{11}-\frac{1}{n}\sum_{j=1}^{n-1}c_{j1} -\frac{1}{n} \sum_{j=1}^{n-1} c_{1j} +\delta, \\
p_{1n}=-\frac{1}{n}\sum_{j=1}^{n-1}c_{1j} +\delta .
\end{gathered}
\end{equation*}
Now,
\[p_{11}-p_{1n}=c_{11}-\frac{1}{n} \sum_{i=1}^{n-1} c_{i1}. \]
By Theorem \ref{rdominant},
\[c_{11} \geq c_{j1}~~\forall j=1,\dotsc,n-1. \]
So,
\[-\frac{1}{n} \sum_{i=1}^{n-1}c_{i1} \geq -(\frac{n-1}{n})  c_{11}.  \]
Hence,
\[c_{11}-\frac{1}{n} \sum_{i=1}^{n-1} c_{i1} \geq c_{11}-\frac{n-1}{n} c_{11} = \frac{1}{n} c_{11} . \]
Since $c_{11} \geq 0$, we conclude that
\[p_{11}-p_{1n} \geq 0. \]
The proof is complete.
\end{proof}

Now it can be easily shown that any generalized resistance is non-negative.
\begin{theorem}
Suppose all the weights in $\g$ are positive scalars. Let $a,b>0$. 
Let $L^\dag=[k_{ij}]$ be the Moore-Penrose inverse of the Laplacian of $\g$. Then,
$$r_{ij}:=a^2 k_{ii} + b^{2} k_{jj} -2ab k_{ij} \geq 0.$$
\end{theorem}
\begin{proof}
We note that the Laplacian matrix $L$ of $\g$ satisfies all the conditions of the previous lemma. 
Moreover, by Proposition $\ref{apd}$, $k_{ii}$ and $k_{jj}$ are positive.
As a consequence of Lemma \ref{lapprop}, we deduce that
\[\min(k_{ii},k_{jj}) \geq \max(k_{ij},k_{ji}).\]
Now 
\[a^2k_{ii} + b^{2} k_{jj}-2ab k_{ij}\geq 0, \]
follows from the arithmetic mean and geometric mean inequality.
\end{proof}

\subsection{A perturbation result}
We now show that if $R$ is the resistance matrix of a connected graph with $n$ vertices,
and if $L$ is the Laplacian matrix of $\g$ with positive scalar weights, then
$(R^{-1}-L)^{-1}$ has all entries non-negative.
\begin{theorem}
Let $H$ be a simple (undirected) connected graph with $n$ vertices
and $R$ be the resistance matrix of $H$.
Assume that all the weights in $\g$ are positive scalars. 
Then, $R^{-1}-L$ is non-singular and every entry in
$(R^{-1}-L)^{-1}$ is non-negative.
\end{theorem}
\begin{proof}
Let $M=[m_{ij}]$ be the Moore-Penrose inverse of the Laplacian matrix of $H$. 
Then the $(i,j)^{\rm th}$ entry $r_{ij}$ of $R$ is given by
\[r_{ij}=m_{ii} + m_{jj}-2m_{ij}. \]
Fix $\a \geq 0$. Define $S:=\a L$.  
We complete the proof by using the following claims.

{\bf Claim 1:} $R^{-1}-S$ is non-singular.\\
To prove this claim, we can assume that $S=\a L$, where $\a>0$.
By Proposition \ref{rankL}, $\rank(S)=n-1$, $S+S'$ is positive semidefinite
and $S' \1=S \1=0$.
Let $x \in \rr^n$ be such that 
\begin{equation} \label{rinverse}
(R^{-1}-S)(x)=0.
\end{equation}
Put $u:=R^{-1}x$.
Assuming $u\neq 0$, 
we now get a contradiction.
As $\1^{'} S=0$, it follows that
$\1' R^{-1}x=0$ and hence
 $u \in \1^\perp$. Writing
\[R=\Diag(M) \1\1' + \1\1' \Diag(M)-2 M, \]
we get
\begin{equation*}
\begin{aligned}
u'Ru &=u'(\Diag(M)\1\1' + \1 \1' \Diag(M)-2M)u \\
&=-2 u'Mu.
\end{aligned}
\end{equation*}
Since $\nulls(M)=\span\{\1\}$, $M$ is positive definite on $\1^{\perp}$. So,
$u'Mu>0$. Hence,
\begin{equation} \label{u'Runegative} 
 u'Ru<0.
 \end{equation}
 It is easy to see that
\begin{equation} \label{u'Ru=x'Rx}
u'Ru=x'R^{-1}x
\end{equation}
By $(\ref{rinverse})$, 
\[x'R^{-1}x=x' S x.\] Since $S+S'$ is positive semidefinite,
$x' Sx \geq 0$. So,
$x'R^{-1}x \geq 0$. Hence by (\ref{u'Ru=x'Rx}),
\begin{equation} \label{u'Rupositive}
u' R u \geq 0.
\end{equation}
Thus, we get a contradiction from
(\ref{u'Runegative}) and (\ref{u'Rupositive}).
 Therefore, $u=R^{-1}x=0$. This implies $x=0$.
So, $R^{-1}-S$ is non-singular.
The claim is proved.

{\bf Claim 2:} If $C$ is a $k \times k$ proper principal submatrix of $S-R^{-1}$,
then \[q' C q >0~~ \mbox{for all} ~0 \neq q \in \rr^{k}. \]
If $A$ is an $n \times n$ matrix, we shall use the notation
$A[i]$ to denote the principal submatrix of $A$ obtained by deleting the
$i^{\rm th}$ row and $i^{\rm th}$ column of $A$.
Fix $1 \leq i \leq n$ and 
define \[B:=S[i] - R^{-1}[i].\]
Since $R$ is negative definite on $\1^{\perp}$
and the diagonal entries are zero, $R$ has exactly one positive eigenvalue.
By an application of interlacing theorem, we see that  
$-R^{-1}[i]$ is positive semidefinite. 
Hence
\begin{equation} \label{-rinversei}
-p'R^{-1} [i]p \geq 0~~\mbox{for all}~w \in \rr^{n-1}.
\end{equation}
As the row sums and the column sums of $S$ are equal to zero, it follows that
all the cofactors of $S$ are equal. Because $S+S'$ is positive semidefinite and
has rank $n-1$, it follows that every proper principal submatrix of $S+S'$ is positive
definite. So, we have 
\begin{equation} \label{p'Li}
p' S [i] p>0 ~~\mbox{for all} ~0 \neq p \in \rr^{n-1}.
\end{equation}
 By (\ref{-rinversei}) and (\ref{p'Li}),
\[p'B p > 0~~\mbox{for all}~~0 \neq p \in \rr^{n-1}.\]
The claim is proved.

In particular, we note that all principal minors of 
$S-R^{-1}$ with order less than $n$ are positive.

{\bf Claim 3:} $\det(S-R^{-1})<0$.\\
Because $\gamma L -R^{-1}$ is non-singular for every $\gamma \geq 0$, 
\[\mbox{sgn}~\det(\gamma L-R^{-1})=\mbox{sgn}~\det(-R^{-1}). \]
Since $-R$ has exactly one negative eigenvalue,
$\det(-R^{-1})<0$.
 So,
\[\det(\gamma L-R^{-1}) <0~~\forall \gamma \geq 0. \]
This proves the claim.

{\bf Claim 4:}
All principal minors of $(S-R^{-1})^{-1}$
are negative. \\
Put \[G:=(S-R^{-1})^{-1}~~\mbox{and}~~H:=S-R^{-1}. \]
Let 
$s < n$, and let
$\widehat{G}$ be a $s \times s$ principal submatrix of $G$.
Suppose $\widehat{H}$
is the complementary submatrix of 
$\widehat{G}$ in $H$.
By Jacobi identity,
\[\det (\widehat{G})=\frac{\det(\widehat{H})}{\det(H)}. \] 
By claim $2$ and $3$, 
\[\mbox{sgn}(\det \widehat{H}) >0~~\mbox{and}~~\mbox{sgn} (\det (H))<0. \]
So,
\(\det(\widehat{G})<0\).
The claim is proved.

We now complete the proof of the theorem.
Given $\beta \geq 0$, let
\[
\left[
\begin{array}{cccc}
y_{ii}(\b) & y_{ij}(\b) \\
y_{ji}(\b) & y_{jj}(\b)
\end{array}
\right]
\]
denote a $2 \times 2$ principal submatrix of $(\b L-R^{-1})^{-1}$. 
By claim $4$, 
\begin{equation} \label{yiibeta}
y_{ii}(\b)<0~~\mbox{and}~~y_{jj}(\b)<0~~\mbox{for all}~\b \geq 0.
\end{equation}
We now show that
$y_{ij}(\b)<0$ for all $\b \geq 0$.
Since $y_{ij}(0)=-r_{ij}$, $y_{ij}(0) <0$.
If $y_{ij}(\a) > 0$ for some $\a>0$, then by continuity, $y_{ij}(\delta)=0$
for some $\delta >0$.
Hence by $(\ref{yiibeta})$,
\[\det(
\left[
\begin{array}{cccc}
y_{ii}(\delta) & y_{ij}(\delta) \\
y_{ji}(\delta) & y_{jj}(\delta) \\
\end{array}
\right])=y_{ii}(\delta) y_{jj}(\delta)>0.
\]
However by claim $4$, 
\[\det(
\left[
\begin{array}{cccc}
y_{ii}(\delta) & y_{ij}(\delta) \\
y_{ji}(\delta) & y_{jj}(\delta) \\
\end{array}
\right])<0.
\]
Thus, we have a contradiction. So, $y_{ij}(\alpha)\leq 0$ for all $\alpha \geq 0$.
We now conclude that every entry in $(L-R^{-1})^{-1}$ is negative. The proof is complete.
\end{proof}

We illustrate the above result with an example.
\begin{example}\rm
Consider the graphs $H$ and $\g$ given in Figure \ref{pertexmm}.
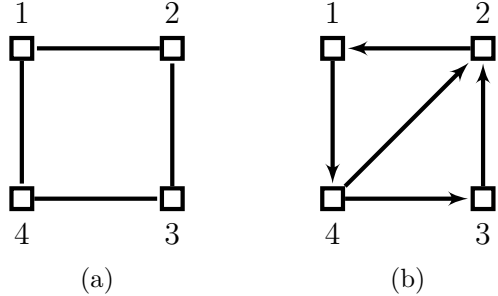
\begin{figure}[tbhp]
\centering
\subfloat[]{\begin{tikzpicture}[shorten >=1pt, auto, node distance=3cm, ultra thick,
   node_style/.style={circle,draw=black,fill=white !20!,font=\sffamily\Large\bfseries},
   edge_style/.style={draw=black, ultra thick}]
\node[label=above:$1$,draw] (1) at  (0,0) {};
\node[label=above:$2$,draw] (2) at  (2,0) {};
\node[label=below:$3$,draw] (3) at  (2,-2) {};
\node[label=below:$4$,draw] (4) at  (0,-2) {};
\draw  (1) to (4);
\draw  (4) to (3);
\draw  (3) to (2);
\draw  (2) to (1); 
\end{tikzpicture}}
~~~~~~~~~~~~\subfloat[]{\begin{tikzpicture}[shorten >=1pt, auto, node distance=3cm, ultra thick,
   node_style/.style={circle,draw=black,fill=white !20!,font=\sffamily\Large\bfseries},
   edge_style/.style={draw=black, ultra thick}]
\node[label=above:$1$,draw] (1) at  (0,0) {};
\node[label=above:$2$,draw] (2) at  (2,0) {};
\node[label=below:$3$,draw] (3) at  (2,-2) {};
\node[label=below:$4$,draw] (4) at  (0,-2) {};
\draw[edge]  (1) to (4);
\draw[edge]  (4) to (3);
\draw[edge]  (3) to (2);
\draw[edge]  (4) to (2);
\draw[edge]  (2) to (1); 
\end{tikzpicture}}
\caption{(a) Graph $H$, and (b) Graph $\g$.}\label{pertexmm}
\end{figure} 
Let the positive scalar weights $w_{ij}$ assigned to each edge $(i,j)$ of $\g$ be
\[w_{14} = w_{21} = \frac{10}{7},~w_{32}=w_{43}=5~\mbox{and}~ w_{42}=2 .\]
The Laplacian matrix of $\g$ is 
\[
L = \frac{1}{10}\left[\begin{array}{rrrr}
     7 & 0 & 0 & -7 \\
-7 & 7 & 0 & 0 \\
0 & -2 & 2 & 0 \\
0 & -5 & -2 & 7
\end{array}\right].\]
The resistance matrix $R$ of $H$ is
\[R = \frac{1}{4}\left[\begin{array}{rrrr}
     0 & 3 & 4 & 3 \\
3 & 0 & 3 & 4 \\
4 & 3 & 0 & 3 \\
3 & 4 & 3 & 0
\end{array}\right].\]
Now,
\[(R^{-1}-L)^{-1} = \frac{1}{8612}\left[\begin{array}{rrrr}
   2335 & 6555 & 7515 & 5125 \\
5125 & 2585 & 6905 & 6915 \\
7515 & 5975 & 1135 & 6905 \\
6555 & 6415 & 5975 & 2585
\end{array}\right],\]
which is a non-negative matrix.
\end{example}

\section*{Funding}
The second author acknowledges the support of the Indian National Science Academy under the INSA Senior Scientist scheme.

\end{document}